\documentclass[a4paper, 12pt]{amsart}

\usepackage{amssymb}
\usepackage{amsmath}
\usepackage{amscd}
\usepackage{graphicx}
\usepackage{amsthm}
\usepackage{color}
\usepackage{enumerate}
\usepackage[all]{xy}

\newtheorem{thm}{Theorem}[section]
\newtheorem{cor}[thm]{Corollary}
\newtheorem{prop}[thm]{Proposition}
\newtheorem{lem}[thm]{Lemma}

\theoremstyle{definition}
\newtheorem{defn}[thm]{Definition}

\newtheorem{rem}[thm]{Remark}

\newcommand{\mf}[1]{{\mathfrak{#1}}}
\newcommand{\mb}[1]{{\mathbf{#1}}}

\begin{document}

%%%%%%%%%%%%%%%%%%%%%%%%%%%%%%%%%%%%%%%%%%%%%%%%%%%%%%%%%%%%%%%
%% Title
%%%%%%%%%%%%%%%%%%%%%%%%%%%%%%%%%%%%%%%%%%%%%%%%%%%%%%%%%%%%%%%
\title{Strongly Koszul edge rings}
%%%%%%%%%%%%%%%%%%%%%%%%%%%%%%%%%%%%%%%%%%%%%%%%%%%%%%%%%%%%%%%
%%%%%%%%%%%%%%%%%%%%%%%%%%%%%%%%%%%%%%%%%%%%%%%%%%%%%%%%%%%%%%% 
%% Information for the first author
%%%%%%% %%%%%%% %%%%%%% %%%%%%% %%%%%%% %%%%%%% %%%%%%% 
\author{Takayuki Hibi}
\address[Takayuki Hibi]{Department of Pure and Applied Mathematics, Graduate School of Information Science and Technology, Osaka University, 
 Toyonaka, Osaka 560-0043, Japan}
\email{hibi@math.sci.osaka-u.ac.jp}
 
%% Information for the second author
%%%%%%% %%%%%%% %%%%%%% %%%%%%% %%%%%%% %%%%%%% %%%%%%% 
\author{Kazunori Matsuda}
\address[Kazunori Matsuda]{Department of Mathematics, College of Science, Rikkyo University, 
 Toshima-ku, Tokyo 171-8501, Japan}
\email{matsuda@rikkyo.ac.jp}
%%%%%%%%%%%%%%%%%%%%%%%%%%%%%%%%%%%%%%%%%%%%%%%%%%%%%%%%%%%%%%% 
%% Information for the third author
%%%%%%% %%%%%%% %%%%%%% %%%%%%% %%%%%%% %%%%%%% %%%%%%% 
\author{Hidefumi Ohsugi}
\address[Hidefumi Ohsugi]{Department of Mathematics, College of Science, Rikkyo University, 
 Toshima-ku, Tokyo 171-8501, Japan}
\email{ohsugi@rikkyo.ac.jp}

%%%%%%%%%%%%%%%%%%%%%%%%%%%%%%%%%%%%%%%%%%%%%%%%%%%%%%%%%%%%%%%
%% General info
%%%%%%%%%%%%%%%%%%%%%%%%%%%%%%%%%%%%%%%%%%%%%%%%%%%%%%%%%%%%%%%
\date{}
%%%%%%%%%%%%%%%%%%%%%%%%%%%%%%%%%%%%%%%%%%%%%%%%%%%%%%%%%%%%%%%

%%%%%%%%%%%%%%%%%%%%%%%%%%%%%%%%%%%%%%%%%%%%%%%%%%%%%%%%%%%%%%%
%% Abstract
%%%%%%%%%%%%%%%%%%%%%%%%%%%%%%%%%%%%%%%%%%%%%%%%%%%%%%%%%%%%%%%

\begin{abstract}
We classify the finite connected simple graphs whose edge rings are strongly Koszul.
From the classification, it follows that if the edge ring is strongly Koszul, then its toric ideal
possesses a quadratic Gr\"obner basis.
\end{abstract}

\maketitle

%%%%%%%%%%%%%%%%%%%%%%%%%%%%%%%%%%%%%%%%%%%%%%%%%%%%%%%%%%%%%%%%%%%%%%%%%%%%%%%%%%%%%%%%%%

\section*{Introduction}
Edge rings, edge polytopes as well as toric ideals arising from finite 
simple graphs (\cite{normal}, \cite{quadgene}, \cite{Koszulbipartite})
are fashionable in the current trend of combinatorics and commutative algebra.
The purpose of the present paper is to classify the finite connected simple graphs $G$ 
for which the edge ring $K[G]$ of $G$ is strongly Koszul (Theorem \ref{main}).
From the classification, it follows that if $K[G]$ is strongly Koszul, then its toric ideal $I_{G}$
possesses a quadratic Gr\"obner basis (Corollary \ref{gbasis}).
It is unclear whether the toric ideal of a strongly Koszul toric ring possesses
a quadratic Gr\"obner basis.

It seems that the class of strongly Koszul toric rings is rather small.
In fact, 
the classification of strongly Koszul algebras has been achieved for 
(i)
edge rings of bipartite graphs (\cite[Theorem 4.5]{HeHiR}),
(ii)
toric rings arising from finite distributive lattices (\cite[Theorem 3.2]{HeHiR}),
and (iii) 
toric rings of stable set polytopes (\cite[Theorem 5.1]{M}).
We refer the reader to, e.g., \cite{BHeV} and \cite{Conca}
for the background of Koszul algebras.

%%%%%%%%%%%%%%%%%%%%%%%%%%%%%%%%%%%%%%%%%%%%%%%%%%%%%%%%%%%%%%%%%%%%%%%%%%%%
% Section 2
% Strongly Koszul algebra
%%%%%%%%%%%%%%%%%%%%%%%%%%%%%%%%%%%%%%%%%%%%%%%%%%%%%%%%%%%%%%%%%%%%%%%%%%%%

\section{Strongly Koszul algebra}

Let $K$ be a field
and let $\mf{m} = R_{+}$ be the homogeneous maximal ideal of 
a graded $K$-algebra $R$. 

\begin{defn}[\cite{HeHiR}]
A graded $K$-algebra $R$ is said to be {\em strongly Koszul} if $\mf{m}$ has a minimal system of generators $\{u_{1}, \ldots, u_{t}\}$
which satisfies the following condition: 

\begin{quote}
For all subsequences $u_{i_{1}}, \ldots, u_{i_{r}}$ of $\{u_{1}, \ldots, u_{t}\}$ $(i_{1} \le \cdots \le i_{r})$ and for all $j = 1, \ldots, r - 1$, 
$(u_{i_{1}}, \ldots, u_{i_{j - 1}}) : u_{i_{j}}$ is generated by a subset of elements of $\{u_{1}, \ldots, u_{t}\}$.
\end{quote}
\end{defn}

A graded $K$-algebra $R$ is called {\em Koszul}
 if $K = R/\mf{m}$ has a linear resolution.
By the following proposition (\cite[Theorem 1.2]{HeHiR}),
we can see that a strongly Koszul algebra is Koszul. 

\begin{prop}
If $R$ is strongly Koszul with respect to the minimal homogeneous generators $\{u_{1}, \ldots, u_{t}\}$ of $\mf{m} = R_{+}$, 
then for all subsequences $\{u_{i_{1}}, \ldots, u_{i_{r}}\}$ of $\{u_{1}, \ldots, u_{t}\}$, $R/(u_{i_{1}}, \ldots, u_{i_{r}})$ has a linear resolution.
\end{prop}

Let $R= K[t_1,\ldots,t_n]$ be the polynomial ring in $n$ variables over $K$.
Let $A$ be a homogeneous affine semigroup ring generated by monomials belonging to
$R$.
If $T$ is a nonempty subset of $[n] = \{1,\ldots, n\}$, then we write 
$R_T$ for the polynomial
ring $K[{t_j : j \in T}]$ with the restricted variables. 
A subring of $A$ of the form $A \cap R_T$
with $\emptyset \neq T \subset [n]$
is called a {\em combinatorial pure subring} of $A$.
See \cite{OHeHi} for details.
%It is known \cite[Corollary 1.6]{OHeHi} that
The following is known to be true \cite[Corollary 1.6]{OHeHi}.

\begin{prop}
\label{cpurecor}
Let $A$ be a homogeneous affine semigroup ring generated by monomials belonging to
$R$
and let $A \cap R_T$ be a combinatorial pure subring of $A$.
If $A$ is strongly Koszul, then so is $A \cap R_T$.
\end{prop}

%%%%%%%%%%%%%%%%%%%%%%%%%%%%%%%%%%%%%%%%%%%%%%%%%%%%%%%%%%%%%%%%%%%%%%%%%%%%
% Section 3
% Strong Koszulness of square-free Veronese subring 
%%%%%%%%%%%%%%%%%%%%%%%%%%%%%%%%%%%%%%%%%%%%%%%%%%%%%%%%%%%%%%%%%%%%%%%%%%%%

%\section{Strong Koszulness of square-free Veronese subring }

This fact is very useful.
% to study strongly Koszul algebras.
For example,
we can determine when the $d$-th squarefree Veronese subring $R_{n, d}$ of
polynomial ring $K[t_1,\ldots,t_n]$ is strongly Koszul.
Note that $R_{n,2}$ is the edge ring of the complete graph $K_n$ of $n$ vertices
(which will be defined later).

\begin{prop}\label{squarefreeVeronese}
Let $2 \leq d < n$ be integers.
Then the following conditions are equivalent:
\begin{enumerate}
 \item[{\rm (i)}] $R_{n, d}$ is strongly Koszul. 
 \item[{\rm (ii)}] Either $(n, d) = (4, 2)$ or $n = d + 1$. 
\end{enumerate}
\end{prop}

\begin{proof}
It is known \cite[Example 1.6 (3)]{HeHiR} that 
$R_{n, 2}$ is strongly Koszul if and only if $n \le 4$.
If $n = d + 1$, then $R_{d + 1, d}$ is isomorphic to a polynomial ring. 
In particular, it is strongly Koszul.

If $(n, d) = (6, 3)$, then the semigroup ring 
$A$ generated by $\{t_1 t_i t_j : 2 \leq i < j \leq 6 \}$
over $K$ is a combinatorial pure subring %(see \cite{OHeHi})
of $R_{6, 3}$.
Since $A$ is isomorphic to $R_{5, 2}$,
$R_{6,3}$ is not strongly Koszul.
By a similar argument, $R_{n,d}$ with $n \ge d + 3$ is not strongly Koszul.
In general, it is known that $R_{n,d} \simeq R_{n,n-d}$.
If $d \ge 3$ and $n = d + 2$, then
$R_{d+2,d} \simeq R_{d+2,2}$  is not strongly Koszul
since $d+2 \geq 5$.
%Hence, $R_{d+2,d}$.
\end{proof}

We call a semigroup ring $A$ {\it trivial}\ \ 
if, starting with polynomial
rings, $A$ is obtained by repeated applications of Segre products and tensor
products.
If $A$ is trivial, then it is strongly Koszul.
See \cite{HeHiR}.

\begin{rem}
\label{remark}
We note that $R_{4,2}$ is a 
{\em non-trivial} strongly Koszul complete intersection semigroup ring. 
Indeed, if we assume that $R_{4, 2}$ is trivial,
then there exists a 3-poset $P$ such that 
the Hibi ring (see \cite{Hib}) $\mathcal{R}_{K}[P]$ is isomorphic to $R_{4, 2}$ 
since all of the trivial rings can be constructed as Hibi rings. 
However, there exists no 3-poset $P$ such that $\mathcal{R}_{K}[P]$ is 
a complete intersection and its embedding dimension is $6$.
Hence, $R_{4,2}$ is not trivial. 
\end{rem}

%\begin{prop}[{[HHR, Corollary 1.5]}]
%\label{degree}
%Let $S$ be a semigroup ring and $K[S]$ the semigroup ring associated with $S$. 
%Let $u$ and $v$ be two monomials of $K[S]$ with $\deg u = k$ and $\deg v = l$ respectively. 
%Assume that $K[S]$ is strongly Koszul. 
%Then the ideal $(u) \cap (v)$ is generated in degree $\le k + l$. 
%\end{prop}

%%%%%%%%%%%%%%%%%%%%%%%%%%%%%%%%%%%%%%%%%%%%%%%%%%%%%%%%%%%%%%%%%%%%%%%%%%%%
% Section 3
% Main Theorem
%%%%%%%%%%%%%%%%%%%%%%%%%%%%%%%%%%%%%%%%%%%%%%%%%%%%%%%%%%%%%%%%%%%%%%%%%%%%

\section{Strongly Koszul edge rings}

Let $G$ be a finite connected simple graph on the vertex set
$V(G) =[n] = \{ 1, 2, \ldots, n \}$.
Let $E(G) = \{ e_{1}, \ldots, e_{d} \}$ be its edge set. 
Recall that a finite graph is {\em simple} if it possesses no loops or multiple edges.
Let $K[X] = K[X_{1}, \ldots, X_{n}]$ be the polynomial ring in $n$ variables
over a field $K$. If $e = \{i, j\} \in E(G)$, then $X^{e}$ stands for the quadratic monomial
$X_{i} X_{j} \in K[X]$. 
The {\em edge ring} of $G$ is the subring 
$K[G] = K[X^{e_{1}}, \ldots, X^{e_{d}}]$ of $K[X]$.
Let $K[Y] = K[Y_{1}, \ldots, Y_{d}]$ denote the polynomial ring in $d$ variables over $K$
with each $\deg Y_{i} = 1$
and define the surjective ring homomorphism $\pi : K[Y] \to K[G]$ by setting
$\pi(Y_{i}) = X^{e_{i}}$ for each $1 \leq i \leq d$.
The {\em toric ideal} $I_{G}$ of $G$ is the kernel of $\pi$.
It is known \cite[Corollary 4.3]{Stu} that $I_{G}$ is generated by binomials
of the form $u - v$, where $u$ and $v$ are monomials of $K[Y]$ with
$\deg u = \deg v$, such that $\pi(u) = \pi(v)$.
In this section, we determine graphs $G$ such that
$K[G]$ is strongly Koszul.

By Proposition \ref{cpurecor}, we have the following fact concerning edge rings.

\begin{cor}
Let $G_{W}$ be an induced subgraph of a graph $G$
on the vertex set $W$.
If $K[G]$ is strongly Koszul, then so is $K[G_{W}]$.
\end{cor}

%\begin{defn}\label{2connecteddef}
Let $G$ be a connected simple graph.
%\begin{enumerate}
% \item[{\rm (i)}] 
Then $G$ is said to be {\em $2$-connected} if $G_{[n] \setminus v}$ is also connected for all $v \in [n]$. 
% \item[{\rm (ii)}] 
Maximal $2$-connected subgraphs of $G$ are called {\em $2$-connected component} of $G$. 
%\end{enumerate}
%\end{defn}
%
If $G$ is bipartite, then the following characterization is known
\cite[Theorem 4.5]{HeHiR}:

\begin{prop}
\label{bpt}
Let $G$ be a connected simple bipartite graph. 
Then $K[G]$ is strongly Koszul if and only if 
any $2$-connected component of $G$ is complete bipartite.
\end{prop}

On the other hand, 
all complete multipartite graphs $G$ 
such that $K[G]$ is strongly Koszul
are classified in \cite[Proposition 3.6]{multipartite}.

Let $G$ be a graph on the vertex set $[n]$.
A {\em walk} of length $q$ of $G$ connecting $v_1\in V(G)$ and $v_{q+1} \in V(G)$ is 
a finite sequence of the form
\begin{equation}
\label{walk}
\Gamma =
 (
\{v_1,v_2\},
\{v_2,v_3\},
\ldots,
\{v_{q},v_{q+1}\}
),
\end{equation}
where $\{v_k,v_{k+1}\} \in E(G)$ for all $k$.
A walk $\Gamma$ is called a {\em path}
if $v_i \neq v_j$ for all $1 \leq i < j \leq q+1$.
An {\em even} (resp.~{\em odd}) {\em walk} is a walk of even (resp.~odd) length.
A walk $\Gamma$ of the form (\ref{walk}) is called {\em closed} if $v_{q+1} =v_1$. 
A {\em cycle} is a closed walk of the form
\begin{equation}
\label{defofcycle}
C = 
 (
\{v_1,v_2\},
\{v_2,v_3\},
\ldots,
\{v_{q},v_{1}\}
),
\end{equation}
where $q \geq 3$ and $v_i \neq v_j$ for all $1 \leq i < j \leq q$.
A {\em chord} of a cycle (\ref{defofcycle}) is an edge $e \in E(G)$ of the form
$e = \{v_i,v_j\}$ for some $1 \leq i < j \leq q$ with $e \notin E(C)$.
If a cycle (\ref{defofcycle}) is even, 
then an {\em even-chord} (resp.~{\em odd-chord}) of (\ref{defofcycle})
is a chord $e = \{v_i,v_j\}$ with $1 \leq i < j \leq q$ 
such that $j-i$ is odd (resp.~even).
A {\em minimal} cycle of $G$ is a cycle having no chords.
If $C_1$ and $C_2$ are cycles of $G$ having no common vertices,
then a {\em bridge} between $C_1$ and $C_2$ is an edge $\{i,j\}$ of $G$
with $i \in V(C_1)$ and $j \in V(C_2)$.

A graph $G$ is called {\em almost bipartite}
if there exists a vertex $v$ such that
any odd cycle of $G$ contains $v$.
Note that the induced subgraph $G_{[n] \setminus v}$ is bipartite.
Let $V_1 \cup V_2$ be a bipartition of $G_{[n] \setminus v}$.
%For an almost bipartite graph $G$,
Then we define the bipartite graph $G(v, V_1,V_2)$
 on the vertex set $[n+1]$ together 
with the edge set
$$ E(G_{[n] \setminus v}) \cup \{ \{i, v\} \in E(G) \ | \ i \in V_1 \} 
\cup \{ \{i, n+1\} \ | \ i \in V_2, \{i, v\} \in E(G) \}.$$
Note that $G$ is obtained from $G(v, V_1,V_2)$ 
by identifying two vertices $v$ and $n+1$. 
In \cite[Proposition 5.5]{csc}, the following is shown.

\begin{prop}
\label{almostbipartite}
Assume the same notation as above.
Then we have $K[G] \simeq K[G(v, V_1,V_2)]$.
\end{prop}

In order to prove the main theorem of this paper,
we need the following lemmata.

\begin{lem}\label{cycle}
Suppose that $K[G]$ is strongly Koszul.
Then $G$ satisfies the following:
\begin{enumerate}
 \item[{\rm (i)}] Every even cycle of $G$ has all possible even-chords. 
 \item[{\rm (ii)}] Any two odd cycles of $G$ have at least two common vertices. 
\end{enumerate}
\end{lem}

\begin{proof}
(i): The following argument is based on \cite{HeHiR}. 
Assume that an even cycle
$$C = (\{1, 2\}, \{2, 3\}, \ldots, \{2k - 1, 2k\}, \{2k, 1\}) \ \ \ (k \ge 3)$$
of $G$ does not have an even-chord $\{1, 2i\}$
for some $2 \le i \le k - 1$. 
Set 
 \[
 u = \prod_{j = 1}^{i - 1} (X_{2j}X_{2j + 1}) \ \ (\deg u = i - 1), \hspace{3mm}
 v = \prod_{j = i}^{k - 1} (X_{2j + 1}X_{2j + 2}) \ \ (\deg v = k - i). 
 \]
Then we can see that $w = \prod_{j = 1}^{2k} X_{i} \in (u) \cap (v)$ and $w / uv = X_{1}X_{2i} \not\in K[G]$. 
Hence $w$ is a generator of $(u) \cap (v)$ with $\deg(w) = k$, but this contradicts
the fact that
$(u) \cap (v)$ is generated in degree $\le k - 1$ (\cite[Corollary 1.5]{HeHiR}). 
Therefore we obtain the desired conclusion. 
 
 (ii): Firstly, we show that there exists no pair of odd cycles having exactly one common vertex. 
Assume that $C_{1}$ and $C_{2}$ are odd cycles 
 \begin{eqnarray*}
 C_{1} &=& (\{1, 2\}, \{2, 3\}, \ldots, \{2i, 2i + 1\}, \{2i + 1, 1\}), \\
 C_{2} &=& (\{2i +1, 2i + 2\}, \{2i + 2, 2i + 3\}, \ldots, \{2m, 2m + 1\}, \{2m + 1, 2i + 1\})
 \end{eqnarray*} 
 $(1 \le i < m)$ 
 of $G$ such that $V(C_{1}) \cap V(C_{2}) = \{2i + 1\}$. 
 Let
 \[
 u = \prod_{j = 1}^{i} (X_{2j - 1}X_{2j}) \ \ (\deg u = i), \hspace{3mm}
 v = \prod_{j = i + 1}^{m} (X_{2j}X_{2j + 1}) \ \ (\deg v = m - i). 
 \]
 Then we can see that $w = X_{2k + 1} \cdot \prod_{j = 1}^{2m + 1} X_{i} \in (u) \cap (v)$ and $w / uv = X_{2k + 1}^{2} \not\in K[G]$, 
 a contradiction as in the proof of (i). 
 
Next, we show that there exists no pair of odd cycles
having no common vertices. 
 Assume that there exist two minimal odd cycles $C_{1}$ and $C_{2}$ having no common vertices. 
By assumption, $K[G]$ is strongly Koszul, and hence $K[G]$ is Koszul. 
 In particular, the toric ideal $I_G$ is generated by quadratic binomials.
Then by \cite[Theorem 1.2]{quadgene},
there exist two bridges $b_{1}$ and $b_{2}$ between $C_{1}$ and $C_{2}$. 
 
 \vspace{2mm}
 
 \hspace{3mm} Case 1: $b_{1}$ and $b_{2}$ have a common vertex. 
 
 \vspace{2mm}
 
 Set $b_{1} = \{a, b\}$ and $b_{2} = \{a, c\}$ $(a \in V(C_{1})$, $b, c \in V(C_{2}))$. 
 Then there exists a path $b \to c$ in $C_{2}$ such that its length is odd. 
 Denote this path by $b \xrightarrow{odd} c$. 
 Then $C = a \xrightarrow{b_{1}} b \xrightarrow{odd} c \xrightarrow{b_{2}} a$ is an odd cycle and 
 $C$ and $C_{1}$ have exactly one common vertex, a contradiction. 
 
 \vspace{2mm}
 
 \hspace{3mm} Case 2: $b_{1}$ and $b_{2}$ have no common vertices. 
 
 \vspace{2mm}

Let $b_{1} = \{a, b\}$ and $b_{2} = \{c, d\}$ 
where $a , c \in V(C_{1})$ and $b, d \in V(C_{2})$. 
%Note that both $C_{1}$ and $C_{2}$ are triangles. 
%Indeed, assume that $\len C_{1} \ge 5$. 
%Then we have that $\len (a \xrightarrow{odd} c) \ge 3$ or $\len (a \xrightarrow{even} c) \ge 4$. 
%We can assume that $\len (a \xrightarrow{odd} c) \ge 3$. 
%Then $C = a \xrightarrow{b_{1}} b \xrightarrow{odd} d \xrightarrow{b_{2}} c \xrightarrow{odd} a$ 
%is an even cycle with $\len C \ge 6$, hence 
%$C$ has all even-chord by (i), this contradicts minimumness of $C_{1}$. 
% Therefore both $C_{1}$ and $C_{2}$ are triangles. 
% 
 Since $a \xrightarrow{b_{1}} b \xrightarrow{even} d \xrightarrow{b_{2}} c \xrightarrow{even} a$ is
an even cycle
of length $\geq 6$, $\{a, d\} \in E(G)$ from (i). 
 Then the distinct two bridges $\{a, b\}$ and $\{a, d\}$ have a common vertex, a contradiction as 
in Case 1. 
 \end{proof}
 
 \begin{lem}\label{2connectedcomp}
Suppose that $K[G]$ is strongly Koszul.
Then at most one $2$-connected component of $G$ is  not complete bipartite. 
 \end{lem}
 
 \begin{proof}
 Assume that $G$ has two distinct non-complete bipartite 2-connected components $G_{1}$, $G_{2}$. 
Then these are not even bipartite from Lemma \ref{cycle} (i), and hence $G_{1}$
(resp. $G_{2})$ has an odd cycle $C_{1}$
(resp. $C_{2})$. 
By Lemma \ref{cycle} (ii), $C_{1}$ and $C_{2}$ have at least two common vertices. 
This contradicts that $C_{1}$ and $C_{2}$ belong to distinct 2-connected components. 
 \end{proof}
 
 \begin{lem}
 \label{K4itself}
Suppose that $K[G]$ is strongly Koszul.
If $G$ has a $2$-connected component $G^{'}$ containing $K_{4}$, then $G^{'}$
is the complete graph $K_{4}$. 
 \end{lem}
 \begin{proof}
Assume that $G^{'} \supsetneq K_{4}$. 
Let $W=\{a,b,c,d\} \subset V(G)$ such that $G'_W= K_{4}$. 
Since $G^{'}$ is connected and $G' \neq G'_W$,
there exists $e \not\in W$ such that $\{a, e\} \in E(G^{'})$. 
Moreover since $G^{'}$ is 2-connected,
there exists a path $P: e \longrightarrow b$ that does not contain
the vertex $a$.
If $P$ contains either $c$ or $d$, then we replace $b$ by it.
Thus, we may assume that $P$ contains neither $c$ nor $d$.
Then $C= a \xrightarrow{} e \xrightarrow{P} b \xrightarrow{} a$ is
a cycle such that $c,d \notin V(C)$. 
 
 If $C$ is an odd cycle, then two odd cycles $(a, c, d)$ and $C$ have exactly one common vertex, a contradiction. 
 Hence $C$ is an even cycle. 
 %Then we have $\{e, f\} \in E(G^{'})$. 
 %Indeed, it is clear if $\len C = 4$. 
 %If $\len C \ge 6$, then $C$ has all even-chords from Lemma \ref{cycle} (i), hence we also have $\{e, f\} \in E(G^{'})$. 
 %Therefore $G^{'}$ contains the following graph Figure 1 as an induced subgraph. 
 Since $G^{'}$ has an even cycle 
$a \xrightarrow{} e \xrightarrow{P} b \xrightarrow{} d \xrightarrow{} c \xrightarrow{} a$
of length $\geq 6$,
we have $\{d, e\} \in E(G)$ by Lemma \ref{cycle} (i). 
 Hence two cycles $(a, b, c)$ and $(a, d, e)$ have exactly one common vertex, 
 but this is a contradiction. 
 Therefore we have $G^{'} = K_{4}$. 
 \end{proof}
\begin{xy}
 \ar@{} (0, 0); (70, 0) *++!D{a} *\dir<4pt>{*} = "A", 
 \ar@{-} "A"; (70, -24) *++!U{b} *\dir<4pt>{*} = "D", 
 \ar@{-} "D"; (60, -8) *++!R{c} *\dir<4pt>{*} = "F", 
 \ar@{-} "D"; (60, -16) *++!R{d} *\dir<4pt>{*} = "E", 
 \ar@{-} "A"; "E",
 \ar@{-} "A"; "F",
 \ar@{-} "E"; "F",
 \ar@{-} "A"; (80, -8) *++!L{e} *\dir<4pt>{*} = "B",
 \ar@{-} "D"; (80, -16) *++!L{ } *\dir<4pt>{*} = "C",
 \ar@{-} "B"; "C", 
 \ar@{} "D";(70, -32) *++!U{\mathrm{Figure}\ 1.},
\end{xy}

\begin{lem}\label{nonalmostbip}
If a graph $G$ is
%2-connected 
strongly Koszul and not almost bipartite, then 
the complete graph $K_{4}$ is a subgraph of
$G$.
\end{lem}

\begin{proof}
Since $G$ is not almost bipartite, 
there exist at least two minimal odd cycles
$C_{1}$ and $C_{2}$.
By Lemma \ref{cycle} (ii), 
$\# \{V(C_{1}) \cap V(C_{2})\} \ge 2$.
Then $C_1$ is divided into several paths
of the form $P = (a_{0}, a_{1}, \ldots, a_{t - 1}, a_{t})$ in $C_{1}$
satisfying that
$a_{0}, a_{t} \in V(C_{2})$ and $a_{i} \not\in V(C_{2})$ $(1 \le i \le t - 1)$.

We will show that $1 \leq t \leq 2$.
Suppose that $t \geq 3$ and $t$ is odd.
Since $C_1$ is minimal, 
the length of a path $a_{t} \xrightarrow{odd\ in\ C_{2}} a_{0}$ 
is not one.
Thus, 
$C= a_{0} \xrightarrow{P} a_{t} \xrightarrow{odd\ in\ C_{2}} a_{0}$ 
is an even cycle of length $\geq 6$.
By Lemma \ref{cycle} (i),
$C$ has all possible even-chords.
This contradict that $C_{1}$ is minimal.
Suppose that $t \geq 4$ and $t$ is even.
Then $C'= a_{0} \xrightarrow{P} a_{t} \xrightarrow{even \ in\ C_{2}} a_{0}$ 
is an even cycle of length $\geq 6$.
By Lemma \ref{cycle} (i),
 $C'$ has all possible even-chords,
and hence either $C_1$ or $C_2$ has a chord.
This is a contradiction.

If $t=1$, then the length of a path $a_{t} \xrightarrow{odd\ in\ C_{2}} a_{0}$ 
is one since $C_2$ is minimal.
If $t=2$, then the length of a path $a_{t} \xrightarrow{even \ in\ C_{2}} a_{0}$ 
is two by the same argument as above.
Thus, $C_1$ and $C_2$ have a common edge
and the lengths of $C_1$ and $C_2$ are equal.
Let $\ell$ be the length of any minimal odd cycle of $G$.

Suppose that $\ell \geq 5$.
Let $C_1 = (1,2,\ldots,\ell)$ and
$\{2,3\}$ be an edge of $C_2$.
Since $G$ is not almost bipartite, there exists
a minimal odd cycle $C_3$ such that $2 \notin V(C_3)$.
Then $C_3$ has edges $\{1,i\}$ and $\{i,3\}$ such that $i \notin V(C_1)$.
Again, since $G$ is not almost bipartite, there exists
a minimal odd cycle $C_4$ such that $3 \notin V(C_4)$.
Then $C_4$ has edges $\{2,j\}$ and $\{j,4\}$ such that $j \notin V(C_1)$.
If $i =j$, then $G$ has a cycle $(2,3,i)$.
This contradicts that $\ell \geq 5$.
Hence, $i \neq j$.
It then follows that $(1,2,j,4,3,i)$ is a cycle of length 6.
By Lemma \ref{cycle} (i), this cycle has even-chord $\{1,4\}$.
This contradicts that $C_1$ is minimal.

Hence $\ell =3$.
Let $C_1 = (1,2,3)$ and $C_2 = (1,2,4)$ be cycles of $G$.
Since $G$ is not almost bipartite, there exists a cycle $C_3$ of 
length 3 with $1 \notin V(C_3)$.
By $\# \{V(C_{1}) \cap V(C_{3})\} \ge 2$
and
$\# \{V(C_{2}) \cap V(C_{3})\} \ge 2$,
we have $C_3 = (2,3,4)$.
Thus, $C_1 \cup C_2 \cup C_3 =K_{4} $, as desired. 
\end{proof}

We now generalize Proposition \ref{bpt} to an arbitrary graph.
The main theorem of this paper is as follows.

\begin{thm}
\label{main}
Let $G$ be a connected simple graph
and let $G_1,\ldots,G_s$ be the 
$2$-connected components of $G$.
Then $K[G]$ is strongly Koszul if and only if 
by a permutation of the indices,
$G_1,\ldots,G_{s-1}$ are complete bipartite and $G_s$ is 
of the following types:
\begin{enumerate}
 \item[{\rm (i)}] A complete bipartite.
\item[{\rm (ii)}] 
An almost bipartite graph such that 
each $2$-connected component of $G_s(v,V_1,V_2)$ is complete bipartite
where $V_1 \cup V_2$ is a bipartition of 
 the vertex set of the bipartite graph
$({G_s})_{[n] \setminus v}$.

\item[{\rm (iii)}] The complete graph $K_{4}$. 
\end{enumerate}
\end{thm}

\begin{proof}
%[Proof of Theorem {\rm \ref{main}}.]
(``If")
It is enough to show that $K[G_s]$ is strongly Koszul.
If $G_s$ satisfies one of (i) and (ii), then $K[G_s]$ is trivial
and hence, in particular, strongly Koszul.
Suppose that $G_s$ satisfies (iii).
By Proposition \ref{squarefreeVeronese}, $R_{4, 2} = K[K_{4}]$ is strongly Koszul.

(``Only if")
Suppose that $K[G]$ is strongly Koszul and
has 2-connected components $G_1,\ldots,G_s$.
By Lemma \ref{2connectedcomp}, 
we may assume that $G_1,\ldots,G_{s-1}$ is complete bipartite.
Suppose that $G_s \neq K_4$ and $G_s$ is not complete bipartite.
If $G_s$ is not almost bipartite, then 
$K_4$ is a subgraph of $G_s$ 
by Lemma \ref{nonalmostbip}.
Then by Lemma \ref{K4itself},
$G_s = K_4$, a contradiction.
Thus, $G_s$ is almost bipartite.
Then by Proposition \ref{almostbipartite},
the bipartite graph $G_s(v,V_1,V_2)$ satisfies $K[G_s] \simeq 
K[G_s(v,V_1,V_2)]$.
By Proposition \ref{bpt}, 
any 2-connected component of $G_s(v,V_1,V_2)$ is complete bipartite.
Therefore, $G_s$ satisfies condition (ii), as desired.
\end{proof}

Since a trivial edge ring is strongly Koszul,
such an edge ring satisfies one of the conditions
in Theorem \ref{main}.
As stated in Remark \ref{remark}, 
$K[K_4] = R_{4,2}$ is nontrivial. 
Hence we have the following.

\begin{cor}\label{trivial}
Let $G$ be a connected simple graph
and let $G_1,\ldots,G_s$ be the 
$2$-connected components of $G$.
Then $K[G]$ is trivial if and only if 
by a permutation of the indices,
$G_1,\ldots,G_{s-1}$ are complete bipartite and $G_s$ is 
of the following types:
\begin{enumerate}
 \item[{\rm (i)}] A complete bipartite.
\item[{\rm (ii)}] 
An almost bipartite graph such that 
each $2$-connected component of $G_s(v,V_1,V_2)$ is complete bipartite
where $V_1 \cup V_2$ is a bipartition of 
 the vertex set of the bipartite graph
$({G_s})_{[n] \setminus v}$.
\end{enumerate}
\end{cor}

It is known (e.g., \cite{Stu}) that
the toric ideals of both $K_4$ and a complete bipartite graph
have a quadratic Gr\"{o}bner basis.
%It then follows from Theorem \ref{main} that
We have the following from Theorem \ref{main}.

\begin{cor}\label{gbasis}
If $K[G]$ is strongly Koszul, then the toric ideal $I_{G}$ 
of $K[G]$ possesses a quadratic Gr\"{o}bner basis. 
\end{cor}

%%%%%%%%%%%%%%%%%%%%%%%%%%%%%%%%%%%%%%%%%%%%%%%%%%%%%%%%%%%%%%%%%%%%
% Acknowledgement
%%%%%%%%%%%%%%%%%%%%%%%%%%%%%%%%%%%%%%%%%%%%%%%%%%%%%%%%%%%%%%%%%%%%

%\par \vspace{2mm}
${\mb{Acknowledgement.}}$ 
This research was supported by the JST (Japan Science and Technology Agency)
CREST (Core Research for Evolutional Science and Technology) research project
Harmony of Gr\"{o}bner Bases and the Modern Industrial Society in the framework of the
JST Mathematics Program ``Alliance for Breakthrough between Mathematics and
Sciences."

%%%%%%%%%%%%%%%%%%%%%%%%%%%%%%%%%%%%%%%%%%%%%%%%%%%%%%%%%%%%%%%%%%%%
% References
%%%%%%%%%%%%%%%%%%%%%%%%%%%%%%%%%%%%%%%%%%%%%%%%%%%%%%%%%%%%%%%%%%%%

\end{document}